\documentclass[12pt,twoside]{article}
\usepackage{titlesec}
\usepackage[a4paper, left=1.2in, right=1.2in, top=0.75in, bottom=0.75in]{geometry}
\usepackage{authblk}
\usepackage{setspace}
\usepackage{geometry}
\usepackage{caption}
\usepackage[utf8]{inputenc}
\usepackage{amsmath}
\usepackage{amssymb, amsthm}
\usepackage{amsfonts}
\usepackage{setspace}
\usepackage{float}
\usepackage{graphicx}
\usepackage[font=small, labelfont=bf]{caption}
\usepackage{hyperref}
\usepackage{orcidlink}
\hypersetup{
   	colorlinks=true,
	linkcolor=blue,
	filecolor=magenta}
\usepackage{fancyhdr}
\usepackage{lipsum} 

\newtheorem{theorem}{Theorem}[section]
\newtheorem{lemma}[theorem]{Lemma}
\newtheorem{corollary}[theorem]{Corollary}
\theoremstyle{definition}

\newtheorem{remark}{Remark}[section]

\numberwithin{equation}{section}  
\usepackage{bigints}
\title{\bf Bohr-Type Inequalities for Fractional Differential and Integral Operators}

\author[1]{Adesanmi Mogbademu}
\author[2]{Ismaila Amusa}

\affil[1]{Department of Mathematics, University of Lagos, Lagos, Nigeria\\
\texttt{amogbademu@unilag.edu.ng}}
\affil[2]{Department of Mathematics, Yaba College of Technology, Lagos, Nigeria\\
\texttt{shesmansecondclass@gmail.com}}

\date{ }

\begin{document}

\maketitle

\vspace{-2cm}

\let\thefootnote\relax

\let\thefootnote\relax

\let\thefootnote\relax
\footnote{\textbf{Corresponding Author:} Ismaila Amusa (ismaila.amusa@yabatech.edu.ng)}

\begin{abstract}
\noindent
We study Bohr-type inequalities within the framework of fractional calculus. Using Riemann–Liouville fractional differential and integral operators, we establish generalized Bohr radii for analytic functions in the unit disk, including the classes of univalent, convex, and Bloch functions. The results unify earlier cases, recover the classical Bohr inequality as a limiting instance, and show how the Bohr radius decreases with the fractional order. Numerical computations further illustrate how the radii vary with the fractional order, highlighting the transition from the classical to the fractional setting.
\end{abstract}

\section{Introduction}
The Bohr inequality, introduced by Harald Bohr \cite{8} in 1914, is a remarkable result in the theory of bounded analytic functions. Let $\mathcal{B}$ denote the class of bounded analytic functions $f(z)=\sum_{n=0}^{\infty}a_nz^n$ in the unit disc $\mathbb{D}:=\{z\in\mathbb{C}: \: |z|<1\}$ such that $|f(z)|<1$ for all $z\in\mathbb{D}$, the Bohr inequality is stated as follows:
\begin{theorem}[Bohr Inequality]
\label{thm1.1}
Suppose that $f(z)=\sum\limits_{n=0}^\infty a_n z^n \in \mathcal{B}$. Then\vspace{-0.2cm}
\begin{equation}
\label{eq1.1}
\sum_{n=0}^\infty |a_n| r^n \leq 1, \quad for \quad r \leq \frac{1}{3}.\vspace{-0.2cm}
\end{equation}
The constant $1/3$ is best possible.
\end{theorem}

\noindent The number ${1}/{3}$ is known as the Bohr radius, and it is sharp. This inequality, though simple in form, has had far-reaching implications in complex analysis, operator theory, and functional analysis.\\
Over the years, the Bohr phenomenon has undergone numerous significant extensions and refinements. In the theory of single complex variable, it has been generalized to cover derivatives, integrals, subordinate functions, harmonic mappings, and functions belonging to special subclasses such as starlike and convex functions (see \cite{1, 2,3,4,7}). Weighted versions, logarithmic refinements, and Bohr-type inequalities in Hardy, Bergman, and Dirichlet spaces have also been established (see \cite{2}). In higher dimensions, the Bohr radius has been studied in the polydisc and the Euclidean unit ball, where its behavior is strikingly different. For instance, Boas and Khavinson \cite{7} showed that the Bohr radius in several variables decays to zero with increasing dimension, in sharp contrast to the classical one-variable case. The asymptotic behavior, refined by Bayart et al. \cite{6} demonstrates that the Bohr radius in $n$ dimensions is equivalent to $\sqrt{(\log n)/n}$.

\noindent  Applications extend to geometric function theory, where coefficient majorization and growth estimates play a central role, as well as to harmonic analysis and problems in several complex variables.\vspace{0.2cm} \\
\noindent  A natural direction of generalization arises when one replaces classical differentiation or integration by their fractional counterparts. Fractional calculus interpolates between differentiation and integration and has proven powerful in both pure and applied mathematics. The study of Bohr-type inequalities for fractional derivatives is motivated by the desire to unify the inequalities known for $f$, $f'$, and the integral operator into a single continuous framework parameterized by the order of differentiation $\alpha$. This investigation not only deepens the understanding of coefficient multiplier inequalities but also opens the way to new applications in geometric function theory, univalent function classes, and harmonic mappings.

\indent The aim of this paper is to establish Bohr-type inequalities for analytic functions acted upon by fractional differential and integral operators. We present general principles, prove specific theorems, and discuss their consequences.

\section{Preliminaries}
Let $H(\mathbb{D})$ denote the class of analytic functions in the unit disc and $H^{\infty}(\mathbb{D})$ the space of bounded analytic functions. Every $f\in H(\mathbb{D})$ admits a power series representation
\[
f(z) = \sum_{n=0}^\infty a_n z^n, \quad z\in \mathbb{D}.
\]

\noindent We recall the definitions of fractional operators relevant to our study. The Riemann–Liouville fractional integral of  order $\alpha$ $(0<\alpha<1)$ is defined by (see \cite{13,15}) 
\begin{equation}
\label{eq2.1}
 I^\alpha f(z) = \frac{1}{\Gamma(\alpha)} \int_0^z (z-\zeta)^{\alpha-1} f(\zeta)\, d\zeta = \sum_{n=0}^{\infty} a_n \frac{\Gamma(n+1)}{\Gamma(n+\alpha+1)} z^{n+\alpha},
\end{equation}
while the corresponding fractional derivative is given by
\begin{equation}
\label{eq2.2}
D^\alpha f(z) =  \frac{1}{\Gamma(1-\alpha)} \frac{d}{dz} \int_0^z (z-\zeta)^{-\alpha} f(\zeta)\, d\zeta =  \sum_{n=0}^{\infty} a_n \frac{\Gamma(n+1)}{\Gamma(n+1-\alpha)} z^{n-\alpha}.
\end{equation}

\noindent In view of \eqref{eq2.1} and \eqref{eq2.2}, we define the Bohr sum (majorant series) associated with $D^\alpha f(z) $ and $I^\alpha f(z)$ as follows:

\begin{equation}
\label{eq2.3}
 B[I^\alpha f] = \sum_{n=0}^{\infty} |a_n| \frac{\Gamma(n+1)}{\Gamma(n+\alpha+1)} r^{n+\alpha},
\end{equation}
 and 
\begin{equation}
\label{eq2.4}
 B[D^\alpha f] = \sum_{n=0}^{\infty} |a_n| \frac{\Gamma(n+1)}{\Gamma(n+1-\alpha)} r^{n-\alpha},
\end{equation}
If $F(z)=\sum_{n=0}^{\infty}b_nz^n$, we say that $F(z)$ satisfies a Bohr Phenomenon  for $R \in (0,1)$, if the inequality
\begin{equation}
\label{eq2.5}
 \sum\limits_{n=1}^{\infty} |b_n|r^n \le \text{dist} \big( F(0), \partial F(\mathbb{D}) \big), \text{ } r=|z|,
\end{equation}
holds for $|z|=r \le R$. Here, $\text{dist}\left( F(0), \partial F(\mathbb{D}) \right)$ denotes the Euclidean distance between $F(0)$ and the boundary $\partial$ of the domain $F(\mathbb{D})$. The largest $R$ is called Bohr radius  (see \cite{2}). Extending this notion to fractional differential and integral transforms of analytic functions forms the subject of this work.\vspace{0.2cm}\\
\noindent In geometric function theory, one often studies the class $\mathcal{S}$ of univalent (i.e.\ injective) holomorphic functions 
\begin{equation}
\label{eq2.6}
f(z)= z+ \sum\limits_{n=2}^{\infty}a_nz^n
\end{equation}
normalized  by $a_0=a_1-1=0$ in $\mathbb{D}$. The Bieberbach conjecture (which was later proved by De Branges\cite{9}) guarantees that the coefficient of each $f \in \mathcal{S}$  satisfy $|a_n| \le n$  for  $n \ge 2$, and the strict inequality holds for all $n$, unless $f$ is the Koebe function $k(z)=z(1-z)^{-2}$ or one of its rotations (see \cite{10}). We also denote the class of convex functions (which is the subclass of univalent) by $\mathcal{C}$. The coefficients of $f\in \mathcal{C}$ satisfy the bound $|a_n|\le1$, $n\ge1$ (see \cite{10}). The following result is well-known in the literature.

\begin{theorem}[see {\cite{10}}] 
\label{thm2.1}
Suppose $f(z)= z+  \sum\limits_{n=1}^{\infty}a_nz^n$, and if
\renewcommand{\theenumi}{\roman{enumi}}
\begin{enumerate}
\item  $f \in \mathcal{S }$, \textit{then } $\dfrac{|z|}{(1+|z|)^2} \le |f(z)| \le \dfrac{|z|}{(1-|z|)^2}$, \quad $z \in \mathbb{D}$
\item $f \in \mathcal{C}$, \textit{then} $\dfrac{|z|}{1+|z|} \le |f(z)| \le \dfrac{|z|}{1-|z|}$, \quad $z \in \mathbb{D}$.
\end{enumerate}
\end{theorem}

\noindent In \cite{2}, the following theorem was established on Bohr inequality for the class of univalent and convex functions:

\begin{theorem}[\cite{2}]
\label{thm2.2}
Let \( f \) and \( g \) be analytic in \( \mathbb{D} \), with \( g \) univalent in \( \mathbb{D} \), and suppose that \( f \prec g \) (i.e., \( f \) is subordinate to \( g \)). Then the inequality \eqref{eq2.5} holds for \( r\le R = 3 - 2\sqrt{2} \). The sharpness of this radius is demonstrated by the Koebe function \( g(z) = z/(1 - z)^{2} \). Moreover, if \( f \) is convex in \( \mathbb{D} \), then the Bohr radius improves to \( R = 1/3 \) .
\end{theorem}

\section{Main Results}
Our main results are organized as follows. In Section~\ref{sec3.1}, we establish Bohr inequalities for the fractional derivative of functions in the class $\mathcal{B}$. Section~\ref{sec3.2} extends the discussion to Bohr-type inequalities for the fractional derivative of univalent functions $f \in \mathcal{S}$ and convex functions $f \in \mathcal{C}$. In Section~\ref{sec3.3}, we determine the Bohr radius for the fractional integral of functions in $\mathcal{B}$. Finally, Section~\ref{sec3.4} is devoted to the Bohr radius for the fractional derivative of bounded analytic Bloch functions. Numerical computations supporting and illustrating the theoretical results are provided in each section.

\subsection{Bohr radius for Fractional Derivative of Analytic Functions}
\label{sec3.1}

\begin{theorem}
\label{thm3.1}
Suppose $f(z)\in \mathcal{B}$ and $f(z)=\sum_{n=0}^\infty a_n z^n$. Then
\begin{equation}
\label{eq3.1}
B[D^\alpha f] = \sum_{n=0}^\infty \frac{\Gamma(n+1)}{\Gamma(n+1-\alpha)} |a_n| r^{n-\alpha} \leq \frac{r^{-\alpha}}{\Gamma(1-\alpha)}, \quad 0<\alpha<1,
\end{equation}
for $r \leq R(\alpha)$, where $R(\alpha)$ is the minimum positive solution of
\[
\sum_{n=1}^\infty \frac{\Gamma(n+1)}{\Gamma(n+1-\alpha)} r^n = \frac{1}{2\Gamma(1-\alpha)}.
\]
\end{theorem}

\begin{proof}
Let $|a_0|=a \in [0, 1)$ and $f\in \mathcal{B}$. Then $|a_n|\le 1-|a_0|^2 = 1-a^2$, $n\ge1$ (see \cite{2}). Therefore,
\[ 
\begin{split}
B[D^\alpha f] &= \sum_{n=0}^\infty |a_n|\frac{\Gamma(n+1)}{\Gamma(n+1-\alpha)} r^{n-\alpha} = |a_0| \frac{r^{-\alpha}}{\Gamma(1-\alpha)} + \sum_{n=1}^\infty |a_n|\frac{\Gamma(n+1)}{\Gamma(n+1-\alpha)} r^{n-\alpha}\\
& \le a\frac{r^{-\alpha}}{\Gamma(1-\alpha)} + (1-a^2)\sum_{n=1}^\infty \frac{\Gamma(n+1)}{\Gamma(n+1-\alpha)} r^{n-\alpha} .
\end{split}
\]
Setting
\begin{equation}
\label{eq3.2}
B[D^\alpha f] \le P(a)= a\frac{r^{-\alpha}}{\Gamma(1-\alpha)} + (1-a^2)\sum_{n=1}^\infty \frac{\Gamma(n+1)}{\Gamma(n+1-\alpha)} r^{n-\alpha}.
\end{equation}
Differentiating $P(a)$ twice w.r.t. $a$ to get
\[ 
P'(a)= \frac{r^{-\alpha}}{\Gamma(1-\alpha)} -2a\sum_{n=1}^\infty \frac{\Gamma(n+1)}{\Gamma(n+1-\alpha)} r^{n-\alpha} 
\]
and 
\[ 
P''(a)= -2\sum_{n=1}^\infty \frac{\Gamma(n+1)}{\Gamma(n+1-\alpha)} r^{n-\alpha}.
\]
It is easy to see that $P''(a)\le0$ for all $a \in [0, 1)$ and $r\in (0, 1)$.  Since $a<1$, obviously,
$P'(a)\ge P'(1)$. That is,
\[ 
P'(a) \ge \frac{r^{-\alpha}}{\Gamma(1-\alpha)} -2\sum_{n=1}^\infty \frac{\Gamma(n+1)}{\Gamma(n+1-\alpha)} r^{n-\alpha} \ge 0,
\]
which holds for $r\in (0,1)$ only if $r\le R(\alpha)$, where $R(\alpha)$ is the minimum positive root of
\[
\frac{r^{-\alpha}}{\Gamma(1-\alpha)} =2\sum_{n=1}^\infty \frac{\Gamma(n+1)}{\Gamma(n+1-\alpha)} r^{n-\alpha}.
\]
Also, $P(a)$ is an increasing function of $a$, for $r\le R(\alpha)$, which implies that $P(a)\le P(1)$. So, from \eqref{eq3.2}
\[
B[D^\alpha f] \le P(a)\le P(1)= \frac{r^{-\alpha}}{\Gamma(1-\alpha)},
\]
for $r\le R(\alpha)$. This completes the proof of Theorem \ref{thm3.1}.
\end{proof}

\begin{remark}
The equality in Theorem \ref{thm3.1} is attained for the function \[f(z)= \dfrac{a-z}{1-az}, \qquad a\in[0,1).\] 
\end{remark}

\begin{remark}
\noindent In the limiting case $\alpha \to 0$, we recover the classical Bohr inequality together with the Bohr radius $1/3$ stated in Theorem \ref{thm1.1}.
\end{remark}

\noindent Table \ref{tab1} gives values of the Bohr radii $R(\alpha)$ for some $\alpha \in [0,1)$.

\begin{table}[H]
\centering
\begin{tabular}{|c|c|c|c|c|c|}
\hline
{\bf $\alpha$} &0&0.2&0.5&0.8&0.99 \\
\hline
{\bf $R(\alpha)$} &0.33333&0.30841&0.28301&0.27026&0.26796\\
\hline
\end{tabular}
\caption{Computed values of Bohr radii $R(\alpha)$ for some $\alpha$-values.}
\label{tab1}
\end{table}

\noindent Although setting $\alpha = 1$ in \eqref{eq2.2} yields the first derivative of $f(z)$, but the Bohr radius does not exist in the limiting case $\alpha \to 1$ of Theorem \ref{thm3.1}. The reason is that $f(z)$ is required to satisfy the condition $f(0)=0$ (see \cite{5} for details) for its derivative $f'(z)$ to have a Bohr phenomenon. This motivates the following:

\begin{theorem}
\label{thm3.2}
Let $f(z)\in \mathcal{B}$ and $f(z)=\sum_{n=0}^\infty a_n z^n$. Then
\begin{equation}
\label{eq3.3}
B[D^\alpha f] = \sum_{n=1}^\infty \frac{\Gamma(n+1)}{\Gamma(n+1-\alpha)} |a_n| r^{n-\alpha} \leq \frac{r^{1-\alpha}}{\Gamma(2-\alpha)}, \quad 0<\alpha<1, 
\end{equation}
for $r \leq \rho(\alpha)$, where $\rho(\alpha)$ is the minimum positive solution of
\[
\sum_{n=2}^\infty \frac{\Gamma(n+1)}{\Gamma(n+1-\alpha)} r^n = \frac{r}{2\Gamma(2-\alpha)}.
\]
\end{theorem}

\begin{proof}
Let $|a_1|=a<1$. Since $f\in \mathcal{B}$, then $|a_n|\le 1-a^2$, $n\ge2$ (see \cite{5}). It follows that
\[ 
\begin{split}
B[D^\alpha f] &= \sum_{n=1}^\infty |a_n|\frac{\Gamma(n+1)}{\Gamma(n+1-\alpha)} r^{n-\alpha} = |a_1| \frac{r^{1-\alpha}}{\Gamma(2-\alpha)} + \sum_{n=2}^\infty |a_n|\frac{\Gamma(n+1)}{\Gamma(n+1-\alpha)} r^{n-\alpha}\\
& \le a\frac{r^{1-\alpha}}{\Gamma(2-\alpha)} + (1-a^2)\sum_{n=2}^\infty \frac{\Gamma(n+1)}{\Gamma(n+1-\alpha)} r^{n-\alpha}.
\end{split}
\]
Setting
\begin{equation}
\label{eq3.4}
B[D^\alpha f] \le \phi(a)= a\frac{r^{1-\alpha}}{\Gamma(2-\alpha)} + (1-a^2)\sum_{n=2}^\infty \frac{\Gamma(n+1)}{\Gamma(n+1-\alpha)} r^{n-\alpha}.
\end{equation}
Differentiation of the function $\phi(a)$ w.r.t. $a$ gives
\[ 
\phi'(a)= \frac{r^{1-\alpha}}{\Gamma(2-\alpha)} -2a\sum_{n=2}^\infty \frac{\Gamma(n+1)}{\Gamma(n+1-\alpha)} r^{n-\alpha} 
\]
and
\[ 
\phi''(a)= -2\sum_{n=2}^\infty \frac{\Gamma(n+1)}{\Gamma(n+1-\alpha)} r^{n-\alpha} \le 0,
\]
for all $a \in [0, 1)$ and $r\in (0, 1)$.  Since $a<1$, it implies that $\phi'(a)\ge \phi'(1)$. Thus,
\[ 
\phi'(a) \ge  \frac{r^{1-\alpha}}{\Gamma(2-\alpha)} -2\sum_{n=2}^\infty \frac{\Gamma(n+1)}{\Gamma(n+1-\alpha)} r^{n-\alpha} \ge 0.\]
Solving for $r\in (0,1)$, we get $r\le \rho(\alpha)$, where $\rho(\alpha)$ is the minimum positive root of
\[
\frac{r^{1-\alpha}}{\Gamma(2-\alpha)} = 2\sum_{n=2}^\infty \frac{\Gamma(n+1)}{\Gamma(n+1-\alpha)} r^{n-\alpha} .
\]
Thus,  $\phi(a)$ is an increasing function of $a$, for $r\le \rho(\alpha)$. It follows that for all $a \in [0, 1)$, $\phi(a)\le \phi(1)$. That is,
\[
B[D^\alpha f] \le \phi(a)\le \phi(1)= \frac{r^{1-\alpha}}{\Gamma(2-\alpha)},
\]
for $r\le \rho(\alpha)$. This completes the proof of Theorem \ref{thm3.2}.
\end{proof}

\begin{remark}
The equality in Theorem \ref{thm3.2} is attained for the function \[f(z)= z \dfrac{a-z}{1-az}, \qquad a\in[0,1).\] 
\end{remark}

\noindent Table \ref{tab2} reveals values of the Bohr radii $\rho(\alpha)$ for some $\alpha \in [0,1)$.

\begin{table}[H]
\centering
\begin{tabular}{|c|c|c|c|c|c|}
\hline
{\bf $\alpha$} &0&0.2&0.5&0.8&1 \\
\hline
{\bf $\rho(\alpha)$} &0.33333&0.30431&0.26004&0.21462&0.18350\\
\hline
\end{tabular}
\caption{Computed values of Bohr radii $\rho(\alpha)$ for some $\alpha$-values.}
\label{tab2}
\end{table}

\begin{remark}
From Table \ref{tab2} we recover the classical Bohr radius at $\alpha=0$. Moreover, the Bohr radius $\rho(\alpha)$ lies between the extremal values at the endpoints of the parameter range:
\[
1/3 \;\le\; \rho(\alpha) \;\le\; 1-\sqrt{2/3}.\]
\end{remark}

\begin{corollary}
\label{cor3.3}
If $f(z) \in \mathcal{B}$ and  $f(z)=\sum_{n=1}^\infty a_n z^n$. Then
\begin{equation}
\label{eq3.5}
B[Df] = \sum_{n=1}^\infty |a_n| nr^{n-1} \leq 1, \quad for \quad r \leq 1- \sqrt{\frac{2}{3}},
\end{equation}
where the constant $1-\sqrt{2/3}$ is best possible.
\end{corollary}

\begin{proof}
By letting $\alpha \to 1$ in Theorem \ref{thm3.2}, we obtain the desired result. In particular, Corollary \ref{cor3.3} recovers Theorem~1 of \cite{5}.
\end{proof}

\noindent It is known from the literature that when the term \( |a_0| \) in Theorem \ref{thm1.1} is replaced by \( |a_0|^2 \), the corresponding Bohr radius reduces to \( 1/2 \) (see \cite{2}). In the following theorem, we extend this result to the setting of fractional derivatives.

\begin{theorem}
\label{thm3.4}
Let $f(z)=\sum_{n=0}^\infty a_n z^n$ with $f(z)\in \mathcal{B}$. Then for $0<\alpha<1$
\begin{equation}
\label{eq3.6}
B[D^\alpha f] =|a_0|^2 \frac{r^{-2\alpha}}{\Gamma^2(1-\alpha)} +  \sum_{n=1}^\infty \frac{\Gamma(n+1)}{\Gamma(n+1-\alpha)} |a_n| r^{n-\alpha} \leq \frac{r^{-2\alpha}}{\Gamma^2(1-\alpha)},
\end{equation}
for $r \leq N(\alpha),$ where $N(\alpha)$ is the minimum positive solution of
\[
\sum_{n=1}^\infty \frac{\Gamma(n+1)}{\Gamma(n+1-\alpha)} r^{n+\alpha} = \frac{1}{\Gamma^2(1-\alpha)}.
\]
\end{theorem}

\begin{proof}
Let $|a_0|=a$ and $f\in \mathcal{B}$. Then $|a_n|\le 1-a^2$, $n\ge1$. Therefore,
\begin{equation}
\label{eq3.7}
 \begin{split}
B[D^\alpha f] & = |a_0|^2 \frac{r^{-2\alpha}}{\Gamma^2(1-\alpha)} + \sum_{n=1}^\infty |a_n|\frac{\Gamma(n+1)}{\Gamma(n+1-\alpha)} r^{n-\alpha}\\
& \le a^2\frac{r^{-2\alpha}}{\Gamma^2(1-\alpha)} + (1-a^2)\sum_{n=1}^\infty \frac{\Gamma(n+1)}{\Gamma(n+1-\alpha)} r^{n-\alpha}.
\end{split}
\end{equation}
Set
\[\phi(a)=a^2\frac{r^{-2\alpha}}{\Gamma^2(1-\alpha)} + (1-a^2)\sum_{n=1}^\infty \frac{\Gamma(n+1)}{\Gamma(n+1-\alpha)} r^{n-\alpha}. \]
Differentiating $\phi(a)$ twice w.r.t. $a$ to have
\[ 
\phi'(a)= 2a\frac{r^{-2\alpha}}{\Gamma^2(1-\alpha)} - 2a\sum_{n=1}^\infty \frac{\Gamma(n+1)}{\Gamma(n+1-\alpha)} r^{n-\alpha} 
\]
and 
\[ 
\phi''(a)= 2\frac{r^{-2\alpha}}{\Gamma^2(1-\alpha)} - 2\sum_{n=1}^\infty \frac{\Gamma(n+1)}{\Gamma(n+1-\alpha)} r^{n-\alpha}.
\]
After few computations, we find that $\phi''(a)\le0$ for all $\alpha \in (0, 1)$ and $r\in [1/2, 1)$. Thus, for every fixed $\alpha \in (0, 1)$ and $r\in [1/2, 1)$,
$$\phi'(a)\ge \phi'(1), \quad \text{for all }\, a<1,$$ 
which implies that
\[ 
\phi'(a)\ge 2\frac{r^{-2\alpha}}{\Gamma^2(1-\alpha)} - 2\sum_{n=1}^\infty \frac{\Gamma(n+1)}{\Gamma(n+1-\alpha)} r^{n-\alpha} \ge0.
\]
This holds for $r\in [1/2,1)$ only if $r\le N(\alpha)$, where $N(\alpha)$ is the minimum positive root of
\[
\frac{r^{-2\alpha}}{\Gamma^2(1-\alpha)} =\sum_{n=1}^\infty \frac{\Gamma(n+1)}{\Gamma(n+1-\alpha)} r^{n-\alpha}.
\]
Again, $\phi(a)$ is an increasing function of $a$. Hence, for $r\le N(\alpha)$, $\phi(a)\le \phi(1)$. So that \eqref{eq3.7} yields
\[
B[D^\alpha f] \le \phi(a)\le \phi(1)= \frac{r^{-2\alpha}}{\Gamma^(1-\alpha)},
\]
for $r\le R(\alpha)$. This completes the proof of Theorem \ref{thm3.4}.
\end{proof}

\begin{table}[H]
\centering
\begin{tabular}{|c|c|c|c|c|c|c|}
\hline
{\bf $\alpha$} &0&0.1&0.2&0.5&0.8&0.9 \\
\hline
{\bf $N(\alpha)$} &0.50000&0.467028&0.431574&0.308621&0.150656&0.083639\\
\hline
\end{tabular}
\caption{Numerical values of $N(\alpha)$ for some selected values of $\alpha\in[0,1)$.}
\label{tab:N-alpha}
\end{table}

Table~\ref{tab:N-alpha} illustrates the behaviour of the Bohr radius $N(\alpha)$. The values were obtained numerically by using the python (google collab). In accordance with the discussion in the text, $N(\alpha)$ is strictly decreasing on $(0,1)$.  In particular, for large fractional orders the admissible radius is very small (e.g.\ $N(0.99)\approx 9.7\times10^{-3}$).

\begin{corollary}
\label{cor3.5}
If $f(z) \in \mathcal{B}$ and $f(z)=\sum_{n=1}^\infty a_n z^n$. Then
\begin{equation}
\label{eq3.8}
|a_0|^2 + \sum_{n=1}^\infty |a_n| r^n \leq 1, \quad for \quad r \leq \frac{1}{2},
\end{equation}
where the constant $1/2$ is best possible.
\end{corollary}

\begin{proof}
The proof of corollary \ref{cor3.5} follows by allowing $\alpha \to 0$ in Theorem \ref{thm3.4}. Furthermore, corollary \ref{cor3.5} recapture exactly the result stated as corollary~3 of \cite{12}.
\end{proof}

\subsection{Bohr radius for Fractional Derivative of Univalent Functions}
\label{sec3.2}

\noindent In this section, we establish Bohr-type inequalities for fractionally differentiated univalent and convex functions. Before stating the theorems, the following lemma is necessary.

\begin{lemma}
\label{lem3.9}
Let $a$ and $b$ be positive integers and $ 0<\alpha<1$. Then
\begin{equation}
\label{eq3.9}
D^\alpha \frac{z^a}{(1+z)^b} = \frac{z^{a-\alpha} \Gamma(a+1)}{\Gamma(a+1-\alpha)}\, {_2F_1}(b,  \,a+1; \, a+1-\alpha; \, -z).
\end{equation}
\end{lemma}

\begin{proof}
Let $G=z^a/(1+z)^b$. From \eqref{eq2.2}, we have
\[
D^\alpha G =  \frac{1}{\Gamma(1-\alpha)} \frac{d}{dz} \int_{0}^{z} \frac{t^a}{(1+t)^b(z-t)^\alpha} dt.
\]
With the transformation $t=zu$, we obtain
\[
D^\alpha G =  \frac{1}{\Gamma(1-\alpha)} \frac{d}{dz} z^{a+1-\alpha} \int_{0}^{1} \frac{u^a}{(1+zu)^b(1-u)^\alpha} du
\]
Using the negative binomial formula for $(1+zu)^{-b}$, we find that 
\[\begin{split}
D^\alpha G &=  \frac{1}{\Gamma(1-\alpha)} \frac{d}{dz} z^{a+1-\alpha}  \sum_{n=0}^{\infty} \binom{b-1+n}{n}(-z)^n   \int_{0}^{1} {u^{a+n}}(1-u)^{-\alpha} du\\
& =  \frac{1}{\Gamma(1-\alpha)} \frac{d}{dz} z^{a+1-\alpha}  \sum_{n=0}^{\infty} \binom{b-1+n}{n}(-z)^n   \beta(a+1+n, \, 1-\alpha)\\
& = z^{a-\alpha}  \sum_{n=0}^{\infty} (b)_n \frac{\Gamma(a+1+n)}{\Gamma(a+1-\alpha+n)} \frac{(-z)^n}{n!} \\
& = \frac{z^{a-\alpha} \Gamma(a+1)}{\Gamma(a+1-\alpha)} \sum_{n=0}^{\infty} \frac{(b)_n(a+1)_n}{(a+1-\alpha)_n} \frac{(-z)^n}{n!}\\
& =  \frac{z^{a-\alpha} \Gamma(a+1)}{\Gamma(a+1-\alpha)}\,  {_2F_1}(b;  \,a+1; \, a+1-\alpha; \, -z).
\end{split}\]
This completes the proof.
\end{proof}

\begin{theorem}
\label{thm3.7}
Let $f(z)=z+\sum_{n=1}^\infty a_n z^n$ be a univalent function in $\mathbb{D}$. Then
\begin{equation}
\label{eq3.10}
B[D^\alpha f] = \sum_{n=1}^\infty \frac{\Gamma(n+1)}{\Gamma(n+1-\alpha)} |a_n| r^{n-\alpha} \leq \mathrm{dist}\big(D^\alpha f(0), \partial D^\alpha f(\mathbb{D}) \big),
\end{equation}
$0<\alpha<1$, for $r \leq K(\alpha)$, where $K(\alpha)$ is the minimum positive solution of
\begin{equation}
\label{eq3.11}
\sum_{n=1}^\infty  \frac{n\, \Gamma(n+1)}{\Gamma(n+1-\alpha)} r^{n-\alpha} = \lim_{|z|\to1}  \left| \frac{z^{1-\alpha}}{\Gamma(2-\alpha)} \: {_2F_1}(2;2;{2-\alpha}; -z) \right|.
\end{equation}
\end{theorem}

\begin{proof}
Let $f\in \mathcal{S}$. From the lower bound of Theorem \ref{thm2.1}(i), and applying \eqref{eq3.9}, we have
\begin{equation}
\label{eq3.12}
\begin{split}
\mathrm{dist}\big(D^\alpha f(0), \partial D^\alpha f(\mathbb{D}) \big) &= \lim_{|z|\to1} \big| \,|D^\alpha f(z)|  -| D^\alpha f(0) | \,\big| = \lim_{|z|\to1} \big|D^\alpha f(z)\big|  \\
&\ge  \lim_{|z|\to1}  \left| \frac{z^{1-\alpha}}{\Gamma(2-\alpha)} \: {_2F_1}(2;2;{2-\alpha}; -z) \right|,
\end{split}
\end{equation}
where $ {_2F_1}(2;2;{2-\alpha}; -z)$ is the hypergeometric function. By De Brange theorem $|a_n|\le n$, $n\ge 2$. Thus, taking the fractional derivative of $f$, we have
\[ 
\begin{split}
B[D^\alpha f] &= \sum_{n=1}^\infty |a_n|\frac{\Gamma(n+1)}{\Gamma(n+1-\alpha)} r^{n-\alpha} = \frac{r^{1-\alpha}}{\Gamma(2-\alpha)} + \sum_{n=2}^\infty |a_n|\frac{\Gamma(n+1)}{\Gamma(n+1-\alpha)} r^{n-\alpha}\\
& \le \frac{r^{1-\alpha}}{\Gamma(2-\alpha)} + \sum_{n=2}^\infty n \frac{\Gamma(n+1)}{\Gamma(n+1-\alpha)} r^{n-\alpha}.
\end{split}
\]
We seek the Bohr radius $r\in(0,1)$ such that $B[D^\alpha f] \le \mathrm{dist}\big(D^\alpha f(0), \partial D^\alpha f(\mathbb{D}) \big)$. Applying \eqref{eq3.12} to have
\[\sum_{n=1}^\infty  \frac{n\, \Gamma(n+1)}{\Gamma(n+1-\alpha)} r^{n-\alpha} \le \lim_{|z|\to1}  \left| \frac{z^{1-\alpha}}{\Gamma(2-\alpha)} \: {_2F_1}(2;2;{2-\alpha}; -z) \right|,\]
which holds in $(0,1)$ only if $r\le K(\alpha)$, where $K(\alpha)$ is given by \eqref{eq3.11}. This completes the proof of Theorem \ref{thm3.7}.
\end{proof}

\begin{remark}
Allowing $\alpha \to 0$, the Bohr radius becomes $r\le K(0)=3-2\sqrt{2}$, which corresponds to the Bohr radius in the classical case as stated in Theorem \ref{thm2.2}. The equality of Theorem \ref{thm3.7} is attained by considering the Koebe function $f(z)=z/(1-z)^2$.
\end{remark}

\noindent Table \ref{tab3} lists the values of the Bohr radii \(K(\alpha)\) for some \(\alpha \in [0,1)\).

\begin{table}[H]
\centering
\begin{tabular}{|c|c|c|c|c|c|c|}
\hline
{\bf $\alpha$} &0&0.1&0.2&0.3&0.4&0.5 \\
\hline
{\bf $K(\alpha)$} &0.171573&0.139068&0.106122&0.07377&0.0439472&0.019832\\
\hline
\end{tabular}
\caption{Computed values of Bohr radii $K(\alpha)$ for some $\alpha$-values.}
\label{tab3}
\end{table}

\begin{theorem}
\label{thm3.8}
Let $f(z)=z+\sum_{n=1}^\infty a_n z^n$ be a convex function in $\mathbb{D}$. Then for $0<\alpha<1$
\begin{equation}
\label{eq3.13}
B[D^\alpha f] = \sum_{n=1}^\infty \frac{\Gamma(n+1)}{\Gamma(n+1-\alpha)} |a_n| r^{n-\alpha} \leq \mathrm{dist}\big(D^\alpha f(0), \partial D^\alpha f(\mathbb{D}) \big),
\end{equation}
for $r \leq P(\alpha)$, where $P(\alpha)$ is the minimum positive solution of
\begin{equation}
\label{eq3.14}
\sum_{n=1}^\infty  \frac{\Gamma(n+1)}{\Gamma(n+1-\alpha)} r^{n-\alpha} = \lim_{|z|\to1}  \left| \frac{z^{1-\alpha}}{\Gamma(2-\alpha)} \: {_2F_1}(1;2;{2-\alpha}; -z) \right|.
\end{equation}
\end{theorem}

\begin{proof}
Let $f\in \mathcal{C}$. It follows from the lower bound of Theorem \ref{thm2.1}(ii), and applying \eqref{eq3.9}, that
\begin{equation}
\label{eq3.15}
\begin{split}
\mathrm{dist}\big(D^\alpha f(0), \partial D^\alpha f(\mathbb{D}) \big) & = \lim_{|z|\to1} \big| |D^\alpha f(z)| - |D^\alpha f(0)|\big|  \\
&\ge  \lim_{|z|\to1}  \left| \frac{z^{1-\alpha}}{\Gamma(2-\alpha)} \: {_2F_1}(1;2;{2-\alpha}; -z) \right|.
\end{split}
\end{equation}
Beacuse $f\in \mathcal{C}$, $|a_n|\le 1$, $n\ge 1$. Thus, we have
\[ 
\begin{split}
B[D^\alpha f] &= \sum_{n=1}^\infty |a_n|\frac{\Gamma(n+1)}{\Gamma(n+1-\alpha)} r^{n-\alpha} = \frac{r^{1-\alpha}}{\Gamma(2-\alpha)} + \sum_{n=2}^\infty |a_n|\frac{\Gamma(n+1)}{\Gamma(n+1-\alpha)} r^{n-\alpha}\\
& \le \frac{r^{1-\alpha}}{\Gamma(2-\alpha)} + \sum_{n=2}^\infty \frac{\Gamma(n+1)}{\Gamma(n+1-\alpha)} r^{n-\alpha}.
\end{split}
\]
From \eqref{eq3.15}, $B[D^\alpha f] \le \mathrm{dist}\big(D^\alpha f(0), \partial D^\alpha f(\mathbb{D}) \big)$ yields
\[\sum_{n=1}^\infty  \frac{ \Gamma(n+1)}{\Gamma(n+1-\alpha)} r^{n-\alpha} \le \lim_{|z|\to1}  \left| \frac{z^{1-\alpha}}{\Gamma(2-\alpha)} \: {_2F_1}(1;2;{2-\alpha}; -z) \right|,\]
which is valid in $(0,1)$ only when $r \leq P(\alpha)$, where $P(\alpha)$ is defined in \eqref{eq3.14}. This concludes the proof.
\end{proof}

\begin{remark}
Letting \(\alpha\to0\) yields \(r\le P(0)={1}/{3}\), thereby recovering the classical Bohr radius (see \cite{12}) for convex function as stated in Theorem~\ref{thm2.2}. The equality of Theorem \ref{thm3.8} is attained by considering the function $f(z)=z/(1-z)$.
\end{remark}

\noindent Table~\ref{tab4} presents the computed values of the Bohr radii \(P(\alpha)\) for various choices of \(\alpha \in (0,1)\).

\begin{table}[H]
\centering
\begin{tabular}{|c|c|c|c|c|c|c|}
\hline
{\bf $\alpha$} &0&0.1&0.2&0.3&0.4&0.5 \\
\hline
{\bf $P(\alpha)$} &0.333333&0.295922&0.255637&0.212347&0.166149&0.117748\\
\hline
\end{tabular}
\caption{Computed values of Bohr radii $P(\alpha)$ for some $\alpha$-values.}
\label{tab4}
\end{table}

\subsection{Bohr radius for Fractional Integral of Analytic Functions}
\label{sec3.3}

\begin{theorem}
Let $f(z)=\sum_{n=0}^\infty a_n z^n$ with $f\in \mathcal{B}$. Then,
\[
\sum_{n=0}^\infty \frac{\Gamma(n+1)}{\Gamma(n+\alpha+1)} |a_n| r^{n+\alpha} \leq \frac{r^\alpha}{\Gamma(1+\alpha)}, \quad 0<\alpha<1,
\]
for $r \leq R_\alpha^{\text{int}}$, where $R_\alpha^{\text{int}}$ is the minimum positive root of
\[
\sum_{n=1}^\infty \frac{\Gamma(n+1)}{\Gamma(n+\alpha+1)} r^n = \dfrac{1}{2\Gamma(1+\alpha)}.
\]
\end{theorem}

\begin{proof}
This result is a direct consequence of \eqref{eq2.1}, and by employing the same reasoning as in the proof of Theorem~\ref{thm3.1}. Therefore, we omit the details here.
\end{proof}

\begin{remark}
Taking the limit \(\alpha\to0\) produces \(r\le R_{0}^{int}= {1}/{3}\), which recovers the classical Bohr radius asserted in Theorem~\ref{thm1.1}.
\end{remark}

\subsection{Bohr radius for Fractional Derivative of Analytic Bloch Functions}
\label{sec3.4}

\noindent An analytic function $f(z)$ on $\mathbb{D}$ is called a Bloch function if 
\[\sup_{z\in \mathbb{D}} |f'(z)| (1-|z|^2) < \infty.\]
The space ${\bf B}$ of Bloch functions is a complex Banach space with the norm $\Vert \bf . \Vert$ defined by the relation (see \cite{14})
\begin{equation}
\label{eq4.1}
\Vert f\Vert_{\bf B} = |f(0)|+ \sup_{z\in \mathbb{D}} |f'(z)| (1-|z|^2).
\end{equation}
Recently, the space ${\bf B}$ together with its various generalizations have been extensively investigated, for instance (see \cite{12,14}). For a detailed study of Bloch functions, we refer the reader to \cite{12,14}. In \cite{12}, the authors derived the Bohr radius for bounded analytic Bloch functions in the unit disk $\mathbb{D}$. Their result is summarized in the following theorem.

\begin{theorem}[{\cite[Theorem 8]{12}}]
\label{thm4.1}
Let $f \in {\bf B}$ with $\|f\|_{\bf B} \leq 1$. Then  
\[
\sum_{n=0}^{\infty} |a_n| r^n \leq 1 \quad \text{for } r \leq R = 0.55356\ldots,
\]
where $R$ is the unique positive solution of
\[
1 - R + R \log(1-R) = 0.
\]
The number R cannot be replaced by a number greater than $0.624162\ldots$.
\end{theorem}
\noindent In this section, we derive the Bohr radius for the fractional derivative of $f \in {\bf B}$, which serves as a generalization of Theorem~\ref{thm4.1}. We formulate the following lemma, which is motivated by arguments appearing in the proof of Theorem 8 in \cite{12}. It will be useful in the proof of our theorem.

\begin{lemma}[{\cite[Proof of Theorem 8]{12}}]
\label{lem4.2}
Let $f(z)=\sum\limits_{n=0}^\infty a_n z^n \in {\bf B}$ with $\|f\|_{\bf B} \leq 1$. Then  
\begin{equation}
\label{eq4.2}
\sum_{n=1}^{\infty} |a_n|^2 r^n \le (1-|a_0|)^2 \ln \frac{1}{1-r},	\quad r<1. 
\end{equation}
\end{lemma}

\begin{theorem}
\label{thm4.3}
Let $f(z)=\sum_{n=0}^\infty a_n z^n \in {\bf B}$ with with $\|f\|_{\bf B} \leq 1$ and let $0<\alpha<1$. Then
\begin{equation}
\label{eq4.3}
B[D^\alpha f] = \sum_{n=0}^\infty \frac{\Gamma(n+1)}{\Gamma(n+1-\alpha)} |a_n| r^{n-\alpha} \leq \frac{r^{-\alpha}}{\Gamma(1-\alpha)}, \quad for \quad r \leq M(\alpha),
\end{equation}
where $M(\alpha)$ is the minimum positive root of
\[
\big({_2F_1}(1,1;2-\alpha;r)\big) \big({_2F_1}(2,1;2-\alpha;r)\big) = \frac{(1-\alpha)^2}{r^2}.
\]
\end{theorem}

\begin{proof}
Observe that
\begin{equation}
\label{eq4.4}
B[D^\alpha f] = |a_0| \frac{r^{-\alpha}}{\Gamma(1-\alpha)} + \sum_{n=1}^\infty |a_n|\frac{\Gamma(n+1)}{\Gamma(n+1-\alpha)} r^{n-\alpha}.
\end{equation}
\noindent Applying the Riemann–Liouville fractional derivative \eqref{eq2.2} to both sides of \eqref{eq4.2}, and following the idea used in the proof of \eqref{eq3.9}, we obtain
\begin{equation}
\label{eq4.5}
\sum_{n=1}^\infty |a_n|^2\frac{\Gamma(n+1)}{\Gamma(n+1-\alpha)} r^{n-\alpha} \le (1-|a_0|)^2 \frac{r^{1-\alpha} }{\Gamma(2-\alpha)}\, {_2F_1}(1;  \,1; \, 2-\alpha; \, r).
\end{equation}
By the classical Cauchy-Schwarz inequality and using \eqref{eq4.5}, we have
\[\begin{split}
\sum_{n=1}^\infty |a_n|\frac{\Gamma(n+1) r^{n-\alpha}}{\Gamma(n+1-\alpha)}  &\le \sqrt{\sum_{n=1}^\infty |a_n|^2 \frac{\Gamma(n+1)r^{n-\alpha}}{\Gamma(n+1-\alpha)} }  \sqrt{\sum_{n=1}^\infty \frac{\Gamma(n+1)r^{n-\alpha}}{\Gamma(n+1-\alpha)} }\\
& \le   \sqrt{(1-|a_0|)^2 \frac{r^{1-\alpha} }{\Gamma(2-\alpha)}\, {_2F_1}(1;  \,1; \, 2-\alpha; \, r)} \\ 
&   \hspace{2cm} \times \sqrt{  \frac{r^{1-\alpha} }{\Gamma(2-\alpha)} \, {_2F_1}(2,1;2-\alpha;r)}\\
& = (1-|a_0|)\,\frac{r^{1-\alpha}}{\Gamma(2-\alpha)} \\
& \quad \times  \sqrt{\, \big({_2F_1}(1,1;2-\alpha;r)\big) \big({_2F_1}(2,1;2-\alpha;r)\big)\,}\\
& = (1-|a_0|) G(r, \alpha),
\end{split}\]
where
\[
G(r, \alpha)= \frac{r^{1-\alpha}}{\Gamma(2-\alpha)} \sqrt{\, \big({_2F_1}(1,1;2-\alpha;r)\big) \big({_2F_1}(2,1;2-\alpha;r)\big)}.
\]
\noindent Consequently, \eqref{eq4.4} reduces to
\begin{equation}
\label{eq4.6}
   B[D^\alpha f] \le  |a_0| \frac{r^{-\alpha}}{\Gamma(1-\alpha)}  + (1-|a_0|)\,G(r,\alpha).
\end{equation}
Furthermore, from \eqref{eq4.3}, we have
\[
|a_0| \frac{r^{-\alpha}}{\Gamma(1-\alpha)}  + (1-|a_0|)\,G(r,\alpha) \le \frac{r^{-\alpha}}{\Gamma(1-\alpha)},
\]
and upon simplification we get
\[G(r,\alpha) \le \frac{r^{-\alpha}}{\Gamma(1-\alpha)}, \quad r\in(0,1).\]
That is,
\begin{equation}
\label{eq4.7}
 \big({_2F_1}(1,1;2-\alpha;r)\big) \big({_2F_1}(2,1;2-\alpha;r)\big) \le \frac{(1-\alpha)^2}{r^2}.
\end{equation}
Solving \eqref{eq4.7} for $r\in(0,1)$, we get $r\le M(\alpha)$, where $M(\alpha)$ is the positive root of $ r^2 \big({_2F_1}(1,1;2-\alpha;r)\big) \big({_2F_1}(2,1;2-\alpha;r)-1\big) = (1-\alpha)^2$. Thus, the proof is complete.
\end{proof}

\begin{remark}
If we allow $\alpha \to 0$, the inequality \ref{eq4.7} reduces to
\begin{equation}
\label{eq4.8}
 \big({_2F_1}(1,1;2-\alpha;r)\big) \big({_2F_1}(2,1;2-\alpha;r)\big) \le \frac{1}{r^2}.
\end{equation}
Using the identities
\[
{}_2F_1(1,1;2;r) = -\frac{\ln(1-r)}{r}, 
\qquad 
{}_2F_1(2,1;2;r) = \frac{1}{1-r},
\]
the ineqality \eqref{eq4.8} simplifies to
\[
-\frac{\ln (1-r)}{1-r} \le \frac{1}{r}, \quad 0<r<1.
\]
This yields ($r\le 0.55356\ldots$) precisely the same Bohr radius as obtained in Theorem~8 of \cite{12} as stated in Theorem \ref{thm4.1}, thereby showing that our fractional result is a genuine extension of the classical setting.
\end{remark}

\begin{table}[H]
\centering
\begin{tabular}{|c|c|c|c|c|c|c|}
\hline
{\bf $\alpha$} &0&0.1&0.2&0.5&0.8&0.9 \\
\hline
{\bf $M(\alpha)$} &0.553567&0.513532&0.471743&0.332954&0.160589&0.088162\\
\hline
\end{tabular}
\caption{Numerical values of the Bohr radii $M(\alpha)$ obtained from the equation \eqref{eq4.8}.}
\label{tab6}
\end{table}

\noindent Table \ref{tab6} shows that $M(\alpha)$ is a decreasing function of $\alpha$, with the largest radius attained at $\alpha=0$ (the classical case), and $M(\alpha)\to 0$ as $\alpha \to 1$. This confirms that the fractional parameter $\alpha$ shrinks the Bohr radius, and the case $\alpha=0$ recovers the result in Theorem~8 of \cite{12}.

\section{Conclusion}
In this paper, we have generalized the Bohr phenomenon to fractional calculus by applying Riemann–Liouville differential and integral operators to analytic, univalent, convex, and Bloch functions. The resulting inequalities reveal how the Bohr radius depends on the fractional order, with the classical Bohr radius recovered as a limiting case. Our computations demonstrate that increasing the fractional parameter reduces the Bohr radius, thereby capturing a smooth transition from classical to fractional calculus. These results not only unify several known inequalities under a single framework but also open new perspectives for extending Bohr-type problems to broader classes of functions and operators in complex and geometric function theory.

\end{document}